\documentclass[12pt,reqno,intlim]{amsart}

\usepackage{amssymb,amsmath}

\usepackage[arrow,matrix,curve]{xy}

\newdir{ >}{{}*!/-5pt/\dir{>}}

\newcommand{\nc}{\newcommand}

\nc{\bC}{\bold{C}} \nc{\bN}{\Bbb{N}} \nc{\cF}{\mathcal{F}}
\nc{\cE}{\mathcal{E}} \nc{\cR}{\mathcal{R}} \nc{\cM}{\mathcal{M}}
\nc{\al}{\alpha} \nc{\bt}{\beta} \nc{\gm}{\gamma} \nc{\dl}{\delta}
\nc{\om}{\omega} \nc{\sg}{\sigma} \nc{\Sg}{\Sigma} \nc{\vf}{\varphi}
\nc{\ve}{\varepsilon} \nc{\os}{\overset} \nc{\ol}{\overline}
\nc{\ul}{\underline} \nc{\us}{\underset} \nc{\sbs}{\subset}
\nc{\bsl}{\backslash} \nc{\Ra}{\Rightarrow}
\nc{\lra}{\longrightarrow} \nc{\all}{\allowdisplaybreaks}

\nc{\Codes}{\operatorname{{\bold{Codes}}}}
\nc{\RegMono}{\operatorname{\mathcal{R}{\rm{eg}\mathcal{M}{\rm{ono}\!}}}}
\nc{\RegEpi}{\operatorname{\mathcal{R}{\rm{eg}\mathcal{E}{\rm{pi}\!}}}}
\nc{\Mn}{\operatorname{\mathcal{M}{\rm{ono}\!}}}
\nc{\Ep}{\operatorname{\mathcal{E}{\rm{pi}\!}}}
\nc{\Rg}{\operatorname{\mathcal{R}{\rm{eg}\!}}}
\nc{\Ob}{\operatorname{Ob\!}}

\numberwithin{equation}{section}

\newtheorem{theo}{\ \ \ Theorem}[section]
\newtheorem{lem}[theo]{\ \ \ Lemma}
\newtheorem{prop}[theo]{\ \ \ Proposition}

\newtheorem{definition}[theo]{\ \ \ Definition}

\theoremstyle{definition}
\newtheorem{exmp}[theo]{\ \ \ Example}

\theoremstyle{remark}
\newtheorem{rem}[theo]{\ \ \ Remark}

\begin{document}

\title[]
{On the implication $T_{0} \Rightarrow T_{3 \frac{1}{2}}$ for some
topological protomodular algebras}

\author{Dali Zangurashvili}

\maketitle

% Abstract.
\begin{abstract}

 The notion of a right-cancellable protomodular algebra is introduced. It
is proved that a right-cancellable topological protomodular algebra
that satisfies the separation axiom $T_{0}$ is completely regular.
\bigskip


\end{abstract}

% 1.
\section{Introduction}

One of the classical results of topological group theory asserts
that any $T_{0}$ group is completely regular (and hence Hausdorff).
There naturally arises the question whether this statement remains
valid if we replace "group" by an algebra from any variety of
universal algebras. A part of this assertion related to the
Hausdorff condition (i.e., that $T_{0}$ implies $T_{2}$) was studied
first by W. Taylor \cite{T}, who proved that this is the case if a
variety is congruence permutable. Subsequently, this result was step
by step improved in the papers
\cite{G},\cite{C},\cite{C1},\cite{KS}. To the best of our knowledge,
the strongest result obtained in this direction asserts that any
$T_{0}$ algebra is Hausdorff if a variety is $k$-permutable and
congruence modular (K. Kearnes, L. Sequeira \cite{KS}).

The topological algebras for another class of varieties were
considered by P. T. Johnstone and M. C. Pedicchio \cite{JP}. Namely,
they studied the category of topological Mal'cev algebras, and, in
particular, proved the implication $T_{0}\Rightarrow T_{2}$ for
them. They also proved that if, in addition, Mal'cev term on an
algebra is weakly associative in the sense of \cite{JP}, then this
algebra is regular. The issue whether the separation axiom $T_{1}$
implies regularity was further studied by F. Borceux and M. M.
Clementino in \cite{BC} and \cite{BC1}, where they generalized many
properties of topological groups to the case of protomodular
algebras, and, in particular, proved that any $T_{1}$ protomodular
algebra is regular. Note that the notion of a protomodular variety
is relatively recent and derived from the notion of a protomodular
category introduced by D. Bourn as an abstract setting in which many
properties of groups remain valid \cite{B}. Note further that there
is the purely syntactical characterization of a protomodular variety
due to D. Bourn and G. Janelidze \cite{BJ}. It requires the
existence of one operation $\theta$ (called a
protomodular/semi-abelian operation) of arbitrarily high arity
$(n+1)$, together with some binary operations $\alpha_{i}$
 and constants $e_{i}$ that satisfy certain identities.

%Later, G. Janelidze, L. Marki and W. Tholen introduced the notion of
%a semi-abelian category as a protomodular category satisfying some
%additional conditions (providing, for instance, that all traditional
%diagrams of homological algebra remain valid) \cite{JMT}. In
%\cite{BJ} D. Bourn and G. Janelidze found the criterion, formulated
%in the syntactical form, for a variety of universal algebras to be
%protomodular/semi-abelian.

The aim of this paper is to establish for which protomodular
algebras the separation axiom $T_{0}$ implies complete regularity.
One relevant sufficient condition is immediate: this is the case if
the algebraic theory of a variety has a group reduct. For this,
there are two criteria: one criterion requires the existence of an
associative Mal'cev operation (Johnstone-Pedicchio \cite{JP}), and
the other requires the existence of a 2-associative semi-abelian
operation \cite{Z}.

%In this way we obtain:\vskip+2mm

%\emph{If the algebraic theory of a variety of universal algebras has
%a 2-associative semi-abelian operation, then any topology on its
%algebra, that satisfies the separation axiom $T_{0}$ is completely
%regular.}\vskip+2mm

Another approach to the problem in question is to attempt to
generalize the well-known result, asserting that any topology on a
group that satisfies the separation axiom $T_{0}$ is determined by a
uniformity, to the case where "group" is replaced by an algebra from
a protomodular variety. Following this approach we distinguish a
certain subvariety. It is determined by the identities

\begin{equation}
\alpha_{i}(\theta(a_{1},a_{2},...,a_{n},b),\theta(a'_{1},a'_{2},...,a'_{n},b))=
\end{equation}
$$=\alpha_{i}(\theta(a_{1},a_{2},...,a_{n},b'),\theta(a'_{1},a'_{2},...,a'_{n},b')),$$

\noindent $1\leq i\leq n$. We call the algebras satisfying these
identities \emph{right-cancellable}. We prove:\vskip+2mm

\emph{A right-cancellable topological protomodular algebra that
satisfies the separation axiom $T_{0}$ is completely regular.}
\vskip+5mm

The results of this paper were announced in \cite{Z1} without
proofs.

%Observe that if the algebra is associative in any sense of \cite{Z},
%then the system of identities (1.1) is equivalent to the system of
%more simple identities

%\begin{equation}
%\alpha_{i}(\theta(a_{1},a_{2},...,a_{n},b),b))=\alpha_{i}(\theta(a_{1},a_{2},...,a_{n},b'),b')),
%\end{equation}
%\vskip+2mm

%\noindent and
%\begin{equation}
%\alpha_{i}(b,\theta(a_{1},a_{2},...,a_{n},b))=\alpha_{i}(b',\theta(a_{1},a_{2},...,a_{n},b')),
%\end{equation}
%\noindent for $1\leq i\leq n$.\vskip+5mm

% 2
\section{Preliminaries}
%We begin with the needed definitions and result from the theory of
%topological protomodular algebras.

%Recall that a quasigroup is a set with three binary operations
%$\cdot, \setminus, /$ satisfying the identities:

%$$a\setminus (ab)=b;$$

%$$a(a\setminus b)=b;$$

%$$(ab)/b=a;$$

%$$(a/b)b=a.$$

%A quasigroup with a unit is called a loop. Elements $1/a$ and
%$a\setminus 1$ are called left and right inverses of $a$. Any loop
%satisfying the associativity property with respect to the
%multiplication is a group.

%The wide class of loops is the one of the so-called IP-loops. Recall
%\cite{} that a loop $A$ is called an IP-loop if the right and left
%inverses
 %of any element $a$ are equal (we denote it by $a^{-1}$), and,
 %moreover, one has
 %$$(ab)b^{-1}=a;$$
 %and

% $$b^{-1}(bc)=c.$$
%\noindent for any $a,b\in A$ \cite{A}. In any IP-loop the identities
%$a/b=ab^{-1}$  and $b\setminus a=b^{-1}a$ are satisfied.

For the definition of a protomodular category we refer the reader to
the paper \cite{B} by D. Bourn.

Let $\mathbb{V}$ be a variety of universal algebras of a type $\cF$.

\begin{theo} (Bourn-Janelidze \cite{BJ})  $\mathbb{V}$ is protomodular if and
only if its algebraic theory contains, for some natural $n$,
constants $e_{1}, e_{2},...,e_{n}$, binary operations $\alpha_{1}$,
$\alpha_{2}$,..., $\alpha_{n}$ and an $(n+1)$-ary operation $\theta$
such that the following identities are satisfied:

\begin{equation}
\alpha_{i}(a,a)=e_{i};
\end{equation}

\begin{equation}
\theta(\alpha_{1}(a,b),\alpha_{2}(a,b),...,\alpha_{n}(a,b),b)=a.
\end{equation}

%\noindent $\mathbb{V}$ is semi-abelian if and only if its signature
%contains a unique constant $e$, and (2.1) and (2.2) are satisfied
%for $e_{1}=e_{2}=...=e_{n}=e$.
\end{theo} \vskip+3mm

For simplicity, algebras from a protomodular variety are called
protomodular. The operation $\theta$ satisfying (2.2) for some
$\alpha_{i}$ and $e_{i}$ which in their turn satisfy (2.1) is called
protomodular. A protomodular operation is called semi-abelian if all
$e_{i}$'s are equal to one another.

The motivating example of a protomodular variety is given by the
variety of groups. More generally, any variety whose algebraic
theory contains a group operation is protomodular (for this reason
the variety of Boolean algebras is protomodular). In that case we
have:

\begin{equation}\theta(a,b)=ab,
\end{equation}

\begin{equation}
\alpha(a,b)=a/b,
\end{equation}

\noindent and $e$ is the unit of the group. The other examples of
protomodular varieties are given by the varieties of left/right
semi-loops, loops, locally Boolean distributive lattices \cite{BC},
Heyting algebras, Heyting semi-lattices \cite{J}. Observe that the
operations (2.3) and (2.4) serve as operations from the
Bourn-Janelidze criterion in the case of left semi-loops and of
loops too.\vskip+2mm

%The variety of semiloops is a protomodular variety. Indeed, take
%$\theta(a,b)=ab$ and $\alpha_{1}(a,b)=a/b$. Other examples are given
%by the varieties of Boolean algebras, Heyting algebras, and any
%variety containing a group operation.

The identities (2.1) and (2.2) immediately imply
\cite{BC1}:\vskip+2mm

(a) if $\alpha_{i}(a,c)=\alpha_{i}(b,c)$, for all $i$ ($1\leq i\leq
n$), then $a=b$;\vskip+2mm

(b) if $\alpha_{i}(a,b)=e_{i}$, for all $i$ ($1\leq i\leq n$), then
$a=b$;\vskip+2mm

(c) $\theta(e_{1},e_{2},...,e_{n},a)=a.$\vskip+2mm

%One can add also the following:

%(d) if $e_{1}=e_{2}=...=e_{n}$, then
%$\theta(a^{-1_{1}},a^{-1_{2}},...a^{-1_{n}},a)=e,$ where
%$a^{-1_{i}}$ denotes $\alpha_{i}(e_{i},a)$.\vskip+2mm

Let $A$ be a set and let $\theta$ be an $(n+1)$-ary operation on
$A$. The two notions of associativity for $\theta$ are introduced in
\cite{Z}. 1-Associativity is the straightforward
       generalization of the usual associativity condition given by
       the move of parentheses.
 $\theta$ is called 2-associative \cite{Z} if, for any
$a_{1},a_{2},...,a_{n}, b_{1},b_{2},...,b_{n},c\in A$, one has
\begin{equation}
\theta(a_{1},a_{2},...,a_{n},\theta(b_{1},b_{2},...,b_{n},c))=
\end{equation}
$$=\theta(\theta(a_{1},a_{2},...,a_{n},b_{1}),\theta(a_{1},a_{2},...,a_{n},b_{2}),...,\theta(a_{1},a_{2},...,a_{n},b_{n}),c).$$\vskip+2mm

%\begin{prop}\cite{Z}
%The algebraic theory of a variety of universal algebras contains a
%group operation if and only if it contains a 2-associative
%semi-abelian operation.
%\end{prop}

Throughout the paper $\mathbb{V}_{n}$ denotes the simplest
protomodular variety, i.e. the variety with the signature
$\mathfrak{F}_{n}$ containing only one $(n+1)$-ary operation symbol
$\theta$, the binary operation symbols
$\alpha_{1},\alpha_{2},...,\alpha_{n}$, and the constant symbols
$e_{1},e_{2},...,e_{n}$,
           where the identities are (2.1) and (2.2).\vskip+2mm

Let now $\mathbb{V}$ be any variety of universal algebras, and let
$A$ be an algebra from $\mathbb{V}$. $A$ is called a topological
algebra if $A$ is equipped with a topology such that all operations
from $\mathcal{F}$ are continuous.

\begin{prop} (Borceux-Clementino \cite{BC1}).
Let $a$ be an element of a topological protomodular algebra $A$.
Then the subsets

\begin{equation}
\bigcap_{i=1}^{n}\alpha_{i}(-,a)^{-1}(H_{i}),
\end{equation}

\noindent with  $H_{i}$ being an open neighborhood of $e_{i}$,
constitute a base of neighborhoods of $a$.
\end{prop}

%Let $X$ be a set. A subset $V$ of $X\times X$ is called a ..., if
%$V$ contains the diagonal $\triangle$ and $V=-V$, where $-V$ is
%defined as ${(y,x)|(x,y)\in V}$.

We refer the reader to \cite{E} for the needed definitions and facts
from the uniform spaces theory. Here we only recall that a
uniformity on a set $X$ is a family of binary relations on $X$, that
satisfies certain conditions. Any uniformity $\mathrm{U}$ determines
a topology on $X$ as follows: $O$ is open if and only if, for any
$x\in O$, there exists $R\in \mathrm{U}$ such that $B(x,R)\subset
O$; here $B(x,R)$ denotes the set $\{y|(x,y)\in R\}$. Moreover, if a
family $\mathbf{C}$ of coverings of a set $X$ satisfies the
conditions (C1)-(C4) below, then the family of sets $(\cup_{H\in
\mathcal{A}}H\times H)_{\mathcal{A}\in \mathbf{C}}$ is a base of a
uniformity on $X$.\vskip+2mm

(C1) If $\mathcal{A}\in \mathbf{C}$ and $\mathcal{A}$ is inscribed
in a covering  $\mathcal{B}$ (i.e., for any $A\in \mathcal{A}$ there
exists $B$ from $\mathcal{B}$, with $A\subset B$), then
$\mathcal{B}\in \mathbf{C}$;\vskip+2mm

(C2) for any $\mathcal{A}_{1},\mathcal{A}_{2}\in \mathbf{C}$, there
is $\mathcal{A}\in \mathbf{C}$, which is inscribed in both
$\mathcal{A}_{1}$ and $\mathcal{A}_{2}$;\vskip+2mm

(C3) for any $\mathcal{A}\in \mathbf{C}$, there exists
$\mathcal{B}\in \mathbf{C}$ which is strongly star-like inscribed in
$\mathcal{A}$ (i.e., for any $B\in \mathcal{B}$ there exists $A$
from $\mathcal{A}$, with $St(B,\mathcal{B})\subset A$);\vskip+2mm

(C4) for any distinct $x,y\in X$, there exists $\mathcal{A}\in
\mathbf{C}$ such that, for any $A\in \mathcal{A}$ we have
$\{x,y\}\nsubseteq \mathcal{A}$.\vskip+2mm

For a set $M$ of $X$ and a covering $\mathcal{A}$ of $X$, the symbol
$St(M, \mathcal{A})$ denotes the set $\bigcup_{A\in \mathcal{A},
A\bigcap M\neq \emptyset}A$.\vskip+3mm

% 2
\section{Right-Cancellable Protomodular Algebras}

\begin{lem}Let $\mathbb{V}$ be a protomodular variety, and let $A$ be a $\mathbb{V}$-algebra. The conditions (i)-(iv) given below are equivalent and imply
 the condition (v).
If $A$ is associative in any sense of \cite{Z}, then all these
conditions are equivalent.\vskip+2mm

{\rm(i)} For any
$a_{1},a_{2},...,a_{n},a'_{1},a'_{2},...,a'_{n},b,b'\in A$ and $i$
($1\leq i\leq n)$, we have

\begin{equation}
\alpha_{i}(\theta(a_{1},a_{2},...,a_{n},b),\theta(a'_{1},a'_{2},...,a'_{n},b))=
%T_{i}(a_{1},a_{2},...,a_{n},a'_{1},a'_{2},...,a'_{n}),
\end{equation}
$$=\alpha_{i}(\theta(a_{1},a_{2},...,a_{n},b'),\theta(a'_{1},a'_{2},...,a'_{n},b'));$$

{\rm(ii)} for any
$a_{1},a_{2},...,a_{n},a'_{1},a'_{2},...,a'_{n},b,b'\in A$ and $i$
($1\leq i\leq n)$, we have

\begin{equation}
\alpha_{i}(\theta(a_{1},a_{2},...,a_{n},\theta{(a'_{1},a'_{2},...,a'_{n},b})),\theta(a''_{1},a''_{2},...,a''_{n},b))=
\end{equation}

$$=\alpha_{i}(\theta(a_{1},a_{2},...,a_{n},\theta{(a'_{1},a'_{2},...,a'_{n},b'})),\theta(a''_{1},a''_{2},...,a''_{n},b'));$$
%T_{i}(a_{1},a_{2},...,a_{n},a'_{1},a'_{2},...,a'_{n}),
\vskip+2mm

{\rm(iii)} for any
$a_{1},a_{2},...,a_{n},a'_{1},a'_{2},...,a'_{n},b,b'\in A$ and $i$
($1\leq i\leq n)$, we have

\begin{equation}
\alpha_{i}(\theta(a_{1},a_{2},...,a_{n},b),\theta(a'_{1},a'_{2},...,a'_{n},\theta(a''_{1},a''_{2},...,a''_{n},b))=
\end{equation}
$$=\alpha_{i}(\theta(a_{1},a_{2},...,a_{n},b'),\theta(a'_{1},a'_{2},...,a'_{n},\theta(a''_{1},a''_{2},...,a''_{n},b')));$$
 \vskip+2mm

{\rm(iv)} for any $i$ ($1\leq i\leq n)$, there is a term $T_{i}$ of
$3n$ variables over the signature of $\mathbb{V}$, such that for any
$a_{1},a_{2},...,a_{n}, a'_{1},a'_{2},...,a'_{n},
a''_{1},a''_{2},...,a''_{n}$, $b, b'\in A,$ if

\begin{equation}
\theta(a_{1},a_{2},...,a_{n},b)=\theta(a'_{1},a'_{2},...,a'_{n},
b'),
\end{equation}

\noindent then

\begin{equation}\alpha_{i}(\theta(a''_{1},a''_{2},...,a''_{n},b'),b)=T_{i}(a_{1},a_{2},...,a_{n},a'_{1},a'_{2},...,a'_{n},a''_{1},a''_{2},...,a''_{n});
\end{equation}\vskip+3mm

%{\rm(ii')} For any $a_{1},a_{2},...,a_{n}, a'_{1},a'_{2},...,a'_{n},
%a''_{1},a''_{2},...,a''_{n}, b, b'\in A,$ if

%\begin{equation}
%\theta(a_{1},a_{2},...,a_{n},b)=\theta(a'_{1},a'_{2},...,a'_{n},
%b'),
%\end{equation}

%\noindent then

%\begin{equation}\alpha_{i}(b',b)=t_{i}(a_{1},a_{2},...,a_{n},a'_{1},a'_{2},...,a'_{n}).
%\end{equation}
%\vskip+3mm

%{\rm(i')} For any $a_{1},a_{2},...,a_{n}, a'_{1},a'_{2},...,a'_{n},
%a''_{1},a''_{2},...,a''_{n}, b, b'\in A,$ if

%\begin{equation}
%\theta(a_{1},a_{2},...,a_{n},b)=\theta(a'_{1},a'_{2},...,a'_{n},
%b'),
%\end{equation}

%\noindent then

%\begin{equation}\alpha_{i}(\theta(a''_{1},a''_{2},...,a''_{n},b'),b)=t_{i}(a_{1},a_{2},...,a_{n},a'_{1},a'_{2},...,a'_{n},a''_{1},a''_{2},...,a''_{n}).
%\end{equation}\vskip+3mm

%(ii) For any ($1\leq i\leq n$), there is a term
%$r_{i}(a_{1},a_{2},...,a_{n})$ such that if

%$$\theta(a_{1},a_{2},...,a_{n},b)=b',$$

%\noindent then

%$$\alpha_{i}(b',b)=r_{i}(a_{1},a_{2},...,a_{n}).$$

%(ii') For any ($1\leq i\leq n$), there is a term
%$r_{i}(a_{1},a_{2},...,a_{n})$ such that if

%$$\theta(a_{1},a_{2},...,a_{n})=b',$$

%\noindent then

%$$\alpha_{i}(\theta(a_{1},a_{2},...,a_{n},b'),b)=r_{i}(a_{1},a_{2},...,a_{n}).$$

%\noindent (and hence is equal to

%$$\alpha_{i}(\theta(a_{1},a_{2},...,a_{n},e_{i}),\theta(a'_{1},a'_{2},...,a'_{n},e_{i})).$$
%For the convenience we denote this term by
%$T_{i}(a_{1},a_{2},...,a_{n},a'_{1},a'_{2},...,a'_{n})$ or just
%$T_{i}$, for any $i$.)

\vskip+3mm

{\rm(v)} for any $a_{1},a_{2},...,a_{n},b,b'\in A$, we have

\begin{equation}
\alpha_{i}(\theta(a_{1},a_{2},...,a_{n},b),b)=
%T_{i}(a_{1},a_{2},...,a_{n},a'_{1},a'_{2},...,a'_{n}),
\end{equation}
$$=\alpha_{i}(\theta(a_{1},a_{2},...,a_{n},b'),b')$$

%\noindent (and hence is equal to

%$$\alpha_{i}(\theta(a_{1},a_{2},...,a_{n},e_{i}),\theta(a'_{1},a'_{2},...,a'_{n},e_{i})).$$
%For the convenience we denote this term by
%$T_{i}(a_{1},a_{2},...,a_{n},a'_{1},a'_{2},...,a'_{n})$ or just
%$T_{i}$, for any $i$.)\vskip+3mm

%{\rm(iii'')} For any
%$a_{1},a_{2},...,a_{n},a'_{1},a'_{2},...,a'_{n},b,b'\in A$, we have

\noindent and

\begin{equation}
\alpha_{i}(b,\theta(a_{1},a_{2},...,a_{n},b))=
%T_{i}(a_{1},a_{2},...,a_{n},a'_{1},a'_{2},...,a'_{n}),
\end{equation}
$$=\alpha_{i}(b',\theta(a_{1},a_{2},...,a_{n},b')).$$

%\noindent (and hence is equal to

%$$\alpha_{i}(\theta(a_{1},a_{2},...,a_{n},e_{i}),\theta(a'_{1},a'_{2},...,a'_{n},e_{i})).$$
%For the convenience we denote this term by
%$T_{i}(a_{1},a_{2},...,a_{n},a'_{1},a'_{2},...,a'_{n})$ or just
%$T_{i}$, for any $i$.)\vskip+3mm

%{\rm(iv)} There are terms $\tau_{i}$ of $2n$ variables such that
%(4.3) implies the following condition:
%\begin{equation}
%b=\theta(\tau_{1}(a_{1},a_{2},...,a_{n},a'_{1},a'_{2},...,a'_{n}),
%\end{equation}
%$$\tau_{2}(a_{1},a_{2},...,a_{n},a'_{1},a'_{2},...,a'_{n}),$$
%$$...,$$
%$$\tau_{n}(a_{1},a_{2},...,a_{n},a'_{1},a'_{2},...,a'_{n}),b').$$

%\vskip+3mm

%{\rm(iv')} There are terms $\tau'_{i}$ of $2n$ variables such that
%the equality
%\begin{equation}
%\theta(a_{1},a_{2},...,a_{n},b)=b'
%\end{equation}

%\noindent implies the following condition:
%\begin{equation}
%b=\theta(\tau'_{1}(a_{1},a_{2},...,a_{n},a'_{1},a'_{2},...,a'_{n}),
%\end{equation}
%$$\tau'_{2}(a_{1},a_{2},...,a_{n},a'_{1},a'_{2},...,a'_{n}),$$
%$$...,$$
%$$\tau'_{n}(a_{1},a_{2},...,a_{n},a'_{1},a'_{2},...,a'_{n}),b').$$

%\vskip+3mm

%{\rm(iv'')} There are terms $\tau''_{i}$ of $2n$ variables such that
%(4.3) implies the following condition:
%\begin{equation}
%b=\theta(\tau''_{1}(a_{1},a_{2},...,a_{n},a'_{1},a'_{2},...,a'_{n}),
%\end{equation}
%$$\tau''_{2}(a_{1},a_{2},...,a_{n},a'_{1},a'_{2},...,a'_{n}),$$
%$$...,$$
%$$\tau''_{n}(a_{1},a_{2},...,a_{n},a'_{1},a'_{2},...,a'_{n}),\theta(a'_{1},a'_{2},...,a'_{n},b')).$$
\end{lem}

\vskip+5mm

\begin{proof}

(i) $\Rightarrow$ (ii): Let

$$c_{i}=\alpha_{i}(\theta(a''_{1},a''_{2},...,a''_{n},b),\theta(a'_{1},a'_{2},...,a'_{n},b)).$$

\noindent By (3.1) the value of $c_{i}$ does not depend on $b$. From
(2.2) we obtain

$$\alpha_{i}(\theta(a_{1},a_{2},...,a_{n},\theta(a'_{1},a'_{2},...,a'_{n},b)),\theta(a''_{1},a''_{2},...,a''_{n},b))=$$
$$=\alpha_{i}(\theta(a_{1},a_{2},...,a_{n},\theta(a'_{1},a'_{2},...,a'_{n},b)),\theta(c_{1},c_{2},...,c_{n},\theta(a'_{1},a'_{2},...,a'_{n},b))).$$\vskip+2mm

%$$=\alpha_{i}(\theta(a_{1},a_{2},...,a_{n},\theta(a'_{1},a'_{2},...,a'_{n},b),$$
%$$\theta(\alpha_{1}(\theta(a''_{1},a''_{2},...,a''_{n},b)),
%\alpha_{2}(\theta(a''_{1},a''_{2},...,a''_{n},b),...,\alpha_{n}(\theta(a''_{1},a''_{2},...,a''_{n},b)),$$
%$$\theta(a'_{1},a'_{2},...,a'_{n},b))=$$
%$$=\alpha_{i}(\theta(a_{1},a_{2},...,a_{n},\theta(a'_{1},a'_{2},...,a'_{n},b'),$$
%$$\theta(\alpha_{1}(\theta(a''_{1},a''_{2},...,a''_{n},b')),
%\theta(\alpha_{2}(\theta(a''_{1},a''_{2},...,a''_{n},b'),...,\theta(\alpha_{n}(\theta(a''_{1},a''_{2},...,a''_{n},b')),$$
%$$\theta(a'_{1},a'_{2},...,a'_{n},b'))=$$
%$$=\alpha_{i}(\theta(a_{1},a_{2},...,a_{n},\theta(a'_{1},a'_{2},...,a'_{n},b'),\theta(a''_{1},a''_{2},...,a''_{n},b')).$$

\noindent Again applying (3.1) we obtain the desired result. One can
prove similarly the implication (i) $\Rightarrow$ (iii).

(ii) $\Rightarrow$ (iv): First note that from (c) of Section 2 we
conclude that (ii) implies (i). Let us now assume that (3.4) is
satisfied. From (2.2) and (3.1) we have
$$\theta(a''_{1},a''_{2},...,a''_{n},b')=\theta(t_{1},t_{2},...,t_{n},\theta(a'_{1},a'_{2},...,a'_{n},b')),$$

\noindent where
$$t_{i}=\alpha_{i}(\theta(a_{1},a_{2},...,a_{n},e_{i}),\theta(a'_{1},a'_{2},...,a'_{n},e_{i})).$$

\noindent Applying (3.4) we obtain
$$\theta(a''_{1},a''_{2},...,a''_{n},b')=\theta(t_{1},t_{2},...,t_{n},\theta(a_{1},a_{2},...,a_{n},b)).$$

\noindent Hence
$$\alpha_{i}(\theta(a''_{1},a''_{2},...,a''_{n},b'),b)=\alpha_{i}(\theta(t_{1},t_{2},...,t_{n},\theta(a_{1},a_{2},...,a_{n},b)),b).$$
\noindent Then from (3.2) we have
$$\alpha_{i}(\theta(a''_{1},a''_{2},...,a''_{n},b'),b)=T_{i}(a_{1},a_{2},...,a_{n},
t_{1},t_{2},...,t_{n})$$

\noindent for the term
$T_{i}=\alpha_{i}(\theta(t_{1},t_{2},...,t_{n},\theta(a_{1},a_{2},...,a_{n},e_{i})),e_{i})$.

(iii) $\Rightarrow$ (iv): From (c) of Section 2 we obtain (3.7). Let
(3.4) be satisfied, and let

$$t_{i}=\alpha_{i}(b,\theta(a_{1},a_{2},...,a_{n},b)).$$

\noindent From (3.7) we have

$$t_{i}=\alpha_{i}(e_{i},\theta(a_{1},a_{2},...,a_{n},e_{i})).$$

\noindent On the other hand, (3.4) and (3.7) imply that

$$t_{i}=\alpha_{i}(b,\theta(a'_{1},a'_{2},...,a'_{n},b')),$$

\noindent and hence, from (2.2) we obtain

$$b=\theta(t_{1},t_{2},...,t_{n},\theta(a'_{1},a'_{2},...,a'_{n},b')).$$

%Applying (5.3) and then (5.9) $n$ times, we obtain

%$$b=\theta(t_{1},t_{2},...,t_{n},\theta(a'_{1},a'_{2},...,a'_{n},b').$$

\noindent Then, taking into account (3.3), we obtain

$$\alpha_{i}(\theta(a''_{1},a''_{2},...,a''_{n},b'),b)=$$
$$=\alpha_{i}(\theta(a''_{1},a''_{2},...,a''_{n},b'),\theta(t_{1},t_{2},...,t_{n},\theta(a'_{1},a'_{2},...,a'_{n},b'))=$$
$$=T_{i}(a_{1},a_{2},...,a_{n},a'_{1},a'_{2},...,a'_{n},a''_{1},a''_{2},...,a''_{n})$$

\noindent for the term
$T_{i}=\alpha_{i}(\theta(a''_{1},a''_{2},...,a''_{n},e_{i}),\theta(t_{1},t_{2},...,t_{n},\theta(a'_{1},a'_{2},...,a'_{n},e_{i}))).$

(iv) $\Rightarrow$ (i): From (c) of Section 2 we have
$$\theta(e_{1},e_{2},...,e_{n},\theta(a'_{1},a'_{2},...,a'_{n},b))=\theta(a'_{1},a'_{2},...,a'_{n},b).$$

\noindent It implies (3.1).

The implication (i) $\Rightarrow$ (v) follows from (c) of Section 2.

(v)  $\Rightarrow$ (i): Let $A$ be associative, and let

$$c_{i}=\alpha_{i}(b,\theta(a'_{1},a'_{2},...,a'_{n},b)).$$

\noindent By (3.7) the value of $c_{i}$ does not depend on $b$. From
(2.2) we have

$$\alpha_{i}(\theta(a_{1},a_{2},...,a_{n},b),\theta(a'_{1},a'_{2},...,a'_{n},b))=$$
$$=\alpha_{i}(\theta(a_{1},a_{2},...,a_{n},\theta(c_{1},c_{2},...,c_{n},\theta(a'_{1},a'_{2},...,a'_{n},b)),\theta(a'_{1},a'_{2},...,a'_{n},b)).$$\vskip+2mm

%$$=\alpha_{i}(\theta(a_{1},a_{2},...,a_{n},$$
%$$\theta(\alpha_{1}(b,\theta(a'_{1},a'_{2},...,a'_{n},b),$$
%$$\alpha_{2}(b,\theta(a'_{1},a'_{2},...,a'_{n},b),...,$$
%$$\alpha_{n}(b,\theta(a'_{1},a'_{2},...,a'_{n},b),$$
%$$\theta(a'_{1},a'_{2},...,a'_{n},b).$$

\noindent Applying the associativity and then (3.6) we obtain the
required equality.
\end{proof}

\begin{definition}
We call a protomodular algebra right-cancellable if it satisfies the
equivalent conditions of Lemma 3.1.
\end{definition}

\begin{rem}(a) Let $n=1$. Below we use the traditional abbreviation $ab$
for $\theta(a,b)$, and $a/b$ for $\alpha(a,b)$.

Let $A$ be an algebra from this variety, and
 the following equivalent conditions be satisfied\footnote{The implication (3.9) $\Rightarrow$(3.8) is obvious. For the converse, observe that
 from (2.2) we have $(a/e)e=a$; on the other hand, (3.8) implies $(a/e)e=a/e.$ The conditions (3.8) and (3.9) are satisfied if, for
instance, $A$ is commutative.}:

\begin{equation}
ae=a,
\end{equation}
\begin{equation}
a/e=a
\end{equation}

\noindent for any $a\in A$. Then the condition $(i)$ implies the
associativity. Indeed, from (3.8) we obtain the identity
\begin{equation}
(ab)/(a'b)=a/a'.
\end{equation}

\noindent It implies

$$(ab)c/(bc)=(ab)/b=ab/eb=a/e=a,$$
\vskip+1mm

\noindent for any $a, b, c\in A$.  Multiplying both parts of this
equality by $(bc)$ from the right, we obtain the associativity
condition. Then, taking into account the fact that any set equipped
with an associative binary operation, which has a left identity and
left inverses, is a group (see, for instance, \cite{F}), we obtain
that $A$ is a group.

Note that (3.1) implies (3.8) and (3.9) in the case of the variety
of left semi-loops. Indeed, we have

$$a/e=aa/ea=aa/a=a.$$
\vskip+1mm

\noindent Thus we can conclude that any right-cancellable left
semi-loop (loop) is a group.

%(b) A Boolean algebra satisfies the condition (i)  if and only if it
%is trivial. Indeed, we have
%$$\alpha_{2}(\theta(a_{1},a_{2},b),b)= a_{1}\wedge (a_{2}\vee \neg b).$$

%\noindent Taking $a_{1}=1$ and $a_{2}=0$, we obtain $\neg b$ on the
%right-hand side of this equality.

(b)  One can show that non-trivial Boolean algebras, locally
 Boolean distributive lattices, Heyting algebras, Heyting
semi-lattices with the protomodular operations given in \cite{BC1}
and \cite{J} are not right-cancellable.

%Similarly, a Heyting algebra $A$ satisfies the condition (i) if and
%only if it is trivial. Indeed, the left-hand side of (5.1) takes the
%form
%\begin{equation}
%(a_{1}\Rightarrow b)\wedge a_{2})\Rightarrow b. \end{equation}
% If
%$b=1$, then it is equal to $1$. Let $b$ be any element of $A$, and
%let $a_{1}=b$, $a_{2}=1$. Then (5.1) is equal to $b$. Hence (i)
%implies that $b=1$.
\end{rem}

The examples of non-trivial right-cancellable algebras are given in
the next section.

% 3
\section{Right-Cancellable Topological Protomodular Algebras}
%As is well-known (see, for instance, \cite{E}), any topology on a
%group that satisfies the separation axiom $T_{0}$ is derived from a
%uniformity. In this section we will generalize this result to the
%case where "group" is replaced by right-cancellable protomodular
%algebras.

Throughout this section, unless specified otherwise, we assume that
$\mathbb{V}$ is a protomodular variety of universal algebras (of a
type $\cF$).

\begin{lem}
For any topological $\mathbb{V}$-algebra, the separation axiom
$T_{0}$ implies $T_{1}$.
%Let $A$ be an $\mathbb{V}$-algebra, and $a\in A$. Assume there is a
%neighbourhood $U$ of $a$, which does not contain $e_{i}$. Then there
%is a neighbourhood $H$ of $e_{i}$ which does not contain $a$.
\end{lem}

\begin{proof} %Consider any $a$ and $b$ in $A$. First consider the
%case where $b=e_{i}$. Let $H$ be a neighbourhood of $e_{i}$ which
%does not contain $a$. Since the mapping $\alpha_{j}(e_{i};-)$ is
%continious and $\alpha_{i}(e_{i},e_{i})=e_{i}$, there is a
%neigbourhood $H'$ of $e_{i}$ such that $\alpha_{i}(e_{i},H')\subset
%H$. Consider the neigbourhood (2.3)
% of $a$
%$$U=\bigcap_{i=1}^{n}\alpha_{j}^{-1}(-,b)(H_{j})$$
%with $H_{i}=H'$ and all $H_{j}$ being neighbourhoods of $e_{j}$.
%Obviously $e_{i}\overline{\in} U$ (otherwise $\alpha_{i}(e_{i},a)\in
%H$).

Let  $A$ be a $\mathbb{V}$-algebra, and let $a,b\in A$. Let there be
a neighbourhood $U$ of $a$ that does not contain $b$. Without loss
of generality one can assume that $U$ is given by (2.3) for some
base $(H_{j})_{1\leq j\leq n}$ of neighbourhoods of $(e_{j})_{1\leq
j\leq n}$. Since $b\overline{\in} U$, there is $j$ such that
$$\alpha_{j}(b,a)\overline{\in} H_{j}.$$

\noindent Since $\alpha_{j}(b;-)$ is continious and
$\alpha_{j}(b,b)=e_{j}$, there is a neighbourhood $H$ of $b$ such
that $\alpha_{j}(b,H)\subset H_{j}$. This implies that
$a\overline{\in} H$.
\end{proof}
\vskip+2mm

\begin{theo} Let $A$ be a right-cancellable $\mathbb{V}$-algebra.
Then any topology on $A$ that satisfies the separation axiom $T_{0}$
is completely regular.
\end{theo}

We will prove this statement in two steps. Let $A$ satisfy the
conditions (i)-(iv) of Lemma 3.1.

\begin{lem} Let $A$ be a
$T_{0}$ $\mathbb{V}$-algebra. Let $\mathcal{B}_{i}$ be any base of
neighbourhoods of $e_{i}$, and let $H_{i}\in {B_{i}}$ $(1\leq i\leq
n)$. Let $H=(H_{1}, H_{2},..., H_{n})$. Consider the covering

$$\mathcal{C}_{H}=(\bigcap_{i=1}^{n}\alpha_{i}^{-1}(-,a))(H_{i}))_{a\in A}$$

\noindent of $A$. Let $\mathbf{C}$ be the family of all coverings of
$A$ in which the coverings $\mathcal{C}_{H}$ are inscribed. Then the
family $\mathbf{C}$ determines a uniformity on $A$.
\end{lem}

\begin{proof}
Let us show that $\mathcal{C}_{H}$ satisfies the conditions
(C1)-(C4).

The validity of (C1) is obvious. The condition (C2) easily follows
from the fact that for any $H'_{i},H''_{i}\in \mathcal{B}_{i}$ there
exists $H_{i}\in \mathcal{B}$ with $H_{i}\subset H'_{i}\bigcap
H''_{i}$.

To prove that the condition (C3) is satisfied, it is sufficient to
show that, for any $H=(H_{1}, H_{2},..., H_{n})\in
\mathcal{B}_{1}\times \mathcal{B}_{2}\times ...\mathcal{B}_{n} $,
there exists $H'=(H'_{1}, H'_{2},...,H'_{n})\in
\mathcal{B}_{1}\times \mathcal{B}_{2}\times ...\mathcal{B}_{n}$ such
that

\begin{equation}
St(\bigcap_{i=1}^{n}\alpha_{i}^{-1}(-,a)(H'_{i}),
\mathcal{C}_{H'})\subset
\bigcap_{i=1}^{n}\alpha_{i}^{-1}(-,a))(H_{i})
\end{equation}

\noindent for any $a\in A$. From the implication (3.4)
$\Rightarrow$(3.5) it follows that
$$T_{i}(e_{1},e_{2},...,e_{n},e_{1},e_{2},...,e_{n},e_{1},e_{2},...,e_{n})=e_{i}$$
\noindent for any $i (1\leq i\leq n)$, and since the mappings
$T_{i}:A^{3n}\rightarrow A$ are continuous, there exist $H_{11}^{i},
H_{12}^{i},H_{13}^{i}\in \mathcal{B}_{1},H_{21}^{i}, H_{22}^{i},
H_{23}^{i}\in \mathcal{B}_{2},...,H_{n1}^{i}, H_{n2}^{i},H_{n3}^{i}$
$\in \mathcal{B}_{n}$ with

\begin{equation}
T_{i}(H^{i}_{11},H^{i}_{21},...H^{i}_{n1},H^{i}_{12},H^{i}_{22},...,H^{i}_{n2},H^{i}_{13},H^{i}_{23},...,H^{i}_{n3})\subset
H_{i}.
\end{equation}

Let

\begin{equation}
H'_{j}=\bigcap_{i=1}^{n}(H^{i}_{j1}\bigcap H^{i}_{j2}\bigcap
H^{i}_{j3}),
\end{equation}

\noindent for any $j$ $(1\leq j\leq n)$. If

$$\bigcap^{n}_{i=1}(\alpha_{i}^{-1}(-,a))(H'_{i})\bigcap \bigcap^{n}_{i=1}(\alpha_{i}^{-1}(-,a'))(H'_{i}))\neq \emptyset,$$

\noindent for some $a'$ from $A$, then there exists $b$ such that
$$\alpha_{i}(b,a)=h_{i}\in H'_{i}$$

\noindent and

$$\alpha_{i}(b,a')=h'_{i}\in H'_{i},$$

\noindent for all $i$ $(1\leq i\leq n)$. We have

$$b=\theta(\alpha_{1}(b,a),\alpha_{2}(b,a),...,\alpha_{n}(b,a),a)=\theta(h_{1},h_{2},...,h_{n},
a)$$

\noindent and

$$b=\theta(\alpha_{1}(b,a'),\alpha_{2}(b,a'),...,\alpha_{n}(b,a'),a')=\theta(h'_{1},h'_{2},...,h'_{n},
a').$$

\noindent Hence
\begin{equation}
\theta(h_{1},h_{2},...,h_{n},a)=\theta(h'_{1},h'_{2},...,h'_{n},a').
\end{equation}

\noindent Let
$$c\in \bigcap^{n}_{i=1}\alpha_{i}^{-1}(-,a')(H'_{i}).$$

\noindent Then $\alpha_{i}(c,a')=h''_{i}\in H'_{i}$ for all $i$
$(1\leq i\leq n)$. We have

$$\alpha_{i}(c,a)=\alpha_{i}(\theta
(\alpha_{1}(c,a'),\alpha_{2}(c,a'),..., \alpha_{n}(c,a'),a'),a)=$$
$$=\alpha_{i}(\theta (h''_{1},h''_{2},...,h''_{n},a'),a).$$

\noindent By the condition (iv) of Lemma 3.1, (4.4) implies that
$$\alpha_{i}(c,a)=T_{i}(h_{1},h_{2},...,h_{n},h'_{1},h'_{2},...,h'_{n},h''_{1}, h''_{2},...,h''_{n}).$$

\noindent From (4.2) it follows that $\alpha_{i}(c,a)\in H_{i}$, and
hence (4.1) is satisfied.

Let us now consider different points $a,a'\in A$. By (b) of Section
2,  there exists $i$ $ (1\leq i\leq n)$ such that

$$\alpha_{i}(a,a')\neq e_{i}.$$

%$$a=\theta
%(\alpha_{1}(a,a'),\alpha_{2}(a,a'),..., \alpha_{n}(a,a'),a')=\theta
%(\alpha_{1}(a',a'),\alpha_{2}(a',a'),...,
%\alpha_{n}(a',a'),a')=a'.$$

\noindent Since $A$ is a $T_{0}$-space, by Lemma 4.1 there exists
$H_{i}\in \mathcal{B}_{i}$ such that
\begin{equation}
\alpha_{i}(a,a')\overline{\in} H_{i}.
\end{equation}

Let
$$t_{i}(a_{1},a_{2},...,a_{n},a'_{1},a'_{2},...,a'_{n})$$ denote the
left (or right)-hand part of (3.1). We obviously have
$$t_{i}(e_{1},e_{2},...,e_{n},e_{1},e_{2},...,e_{n})=e_{i}.$$
Since $t_{i}:A^{2n}\rightarrow A$ is continuous, there are
$$\widetilde{H}_{11},\widetilde{H}_{12}\in
\mathcal{B}_{1},\widetilde{H}_{21},\widetilde{H}_{22}\in
\mathcal{B}_{2},...,\widetilde{H}_{n1},\widetilde{H}_{n2}\in
\mathcal{B}_{n}
$$

\noindent with
\begin{equation}
t_{i}(\widetilde{H}_{11},\widetilde{H}_{21},...,\widetilde{H}_{n1},\widetilde{H}_{12},\widetilde{H}_{22},...,\widetilde{H}_{n2}
)\subset H_{i}.
\end{equation}

\noindent Let
\begin{equation}
\widetilde{H}_{j}=\widetilde{H}_{j1}^{i}\bigcap
\widetilde{H}_{j2}^{i}
\end{equation}

\noindent for any $j$ $(1\leq j\leq n)$, and let

$$\widetilde{H}=(\widetilde{H}_{1},\widetilde{H}_{2},...,\widetilde{H}_{n}).$$

\noindent Let us show that none of the elements of
$\mathcal{C}_{\widetilde{H}}$ contains both $a$ and $a'$. Indeed,
suppose that $$a,a'\in
\bigcap^{n}_{j=1}(\alpha_j^{-1}(-,b))(\widetilde{H}_{j}).$$

\noindent Then
$$\alpha_{j}(a,b)=h_{j}\in \widetilde{H}_{j}$$
\noindent and
$$\alpha_{j}(a',b)=h'_{j}\in \widetilde{H}_{j},$$
\noindent for all $j$'s. We have
$$a=\theta (\alpha_{1}(a,b),\alpha_{2}(a,b),...,
\alpha_{n}(a,b),b)=\theta (h_{1},h_{2},..., h_{n},b).$$

\noindent Similarly,
$$a'=\theta (h'_{1},h'_{2},..., h'_{n},b).$$

\noindent Then, by the condition (i) of Lemma 3.1, we have

$$\alpha_{i}(a,a')=\alpha_{i}(\theta(h_{1},h_{2},...,h_{n},b),\theta(h'_{1},h'_{2},...,h'_{n},b))=$$
$$=t_{i}(h_{1},h_{2},...,h_{n},h'_{1},h'_{2},...,h'_{n}).$$

\noindent From (4.6) and (4.7) it follows that

$$\alpha_{i}(a,a')\in H_{i},$$

\noindent but this contradicts (4.5).
\end{proof}

\begin{lem}
The topology $\tau'$ induced by the uniformity described in Lemma
4.3 coincides with the topology $\tau$ of $A$.
\end{lem}

\begin{proof}
First note that the family of all binary relations
$$(\bigcup_{a\in A}
\bigcap_{i=1}^{n}(\alpha_{i}^{-1}(-,a)(H_{i}))^{2})_{H_{1}\in
\mathcal{B}_{1}, H_{2}\in \mathcal{B}_{2},..., H_{n}\in
\mathcal{B}_{n} },$$ \noindent is the base of the uniformity
determined by $\mathbf{C}$ (see Section 2). A subset $O$ of $A$ is
open in the topology $\tau'$ induced by this uniformity if and only
if for any $a\in O$ there exist $H_{1}\in \mathcal{B}_{1}$,
$H_{2}\in \mathcal{B}_{2}$,..., $H_{n}\in \mathcal{B}_{n}$ such that

$$St(a,\mathcal{C}_{H})\subset O$$

\noindent for $H=(H_{1},H_{2},...,H_{n})$. Obviously,
$\bigcap_{i=1}^{n}\alpha_{i}^{-1}(-,a)(H_{i})$ is open in $\tau$ and
is contained in $St(a, \mathcal{C}_{H})$. Therefore any subset $O$
being open in $\tau'$ is also open in $\tau$. For the converse,
consider a subset $O$ of $A$ which is open in $\tau$. Let $a\in O$.
According to Proposition 2.1, there exist $H_{1}\in
\mathcal{B}_{1}$, $H_{2}\in \mathcal{B}_{2}$,..., $H_{n}\in
\mathcal{B}_{n}$ such that
$$\bigcap_{i=1}^{n}\alpha_{i}^{-1}(-,a)(H_{i})\subset O.$$
\noindent But, as shown in the proof of Lemma 4.3, there exists
$(H'_{1},H'_{2},...,H'_{n})\in \mathcal{B}_{1}\times
\mathcal{B}_{2}\times \mathcal{B}_{n}$ such that (4.1) holds. Since
$$St(a,\mathcal{C}_{H})\subset
St(\bigcap_{i=1}^{n}\alpha_{i}^{-1}(-,a))(H'_{i}),
\mathcal{C}_{H'}),$$
\noindent $O$  is open in $\tau'$.

\end{proof}

Lemma 4.3 and Lemma 4.4 obviously imply Theorem 4.2.

\begin{exmp}  Let $A=\{0,1\}$. Let us introduce the structure of a $\mathbb{V}_{2}$-algebra on $A$ as
follows. Let $\theta(i,j,k)=k$ if $i\neq j$ and $\theta(i,j,k)=1-k$
if $i=j$. Moreover, let $\alpha_{1}(i,j)$ be $0$, for any $i,j$; let
$\alpha_{2}(i,j)$ be 0 if $i\neq j$, and be 1 if $i=j$. Besides, let
$e_{1}=0$ and $e_{2}=1$.

The condition (3.1) is obviously satisfied for $i=1$. It is also
satisfied for $i=2$ since the value of the left-hand side of (3.1)
depends only on whether $a_{1}$ and $a'_{1}$ are equal respectively
to $a_{2}$ and $a'_{2}$. This implies that $A$ is a
right-cancellable algebra.

This example in itself can not be considered as an example
illustrating Theorem 4.2, since any topology on $A$ that satisfies
$T_{1}$ is discrete. However, this example gives rise a lot of
contensive ones. Indeed, for any set $I$, and any congruence $R$ on
$(\prod_{I}A)$, the algebra $(\prod_{I}A)/R$ obviously also lies in
the variety of right-cancellable $V_{2}$-algebras. By Theorem 4.2,
any topology on $(\prod_{I}A)/R$ that satisfies $T_{0}$ is
completely regular.
\end{exmp}

%Finally let us give

%\begin{theo}
%If the algebraic theory of a variety of universal algebras has a
%2-associative semi-abelian operation, then all topologies on its
%algebras that satisfy the separation axiom $T_{0}$ are completely
%regular.
%\end{theo}

%\begin{proof}The statement immediately follows from Proposition 2.2.
%\end{proof}

%\begin{theo}
%Let $A$ be a $\mathbb{V}$-algebra satisfying any of associativity
%conditions introduced in Section 3. Assume also that, for any $i$
%$(1\leq i\leq n)$, there exists a term
%$r_{i}(a_{1},a_{2},...,a_{2n})$ satisfying the following condition:

%\begin{equation}
%\alpha_{i}(\theta(a_{1},a_{2},...,a_{n},b),b')=
%\end{equation}
%$$=r_{i}(a_{1},a_{2},...,a_{n},\alpha_{1}(b,b'),\alpha_{2}(b,b'),...,\alpha_{n}(b,b')).$$

%\noindent Then any $T_{0}$ topology on $A$ is completely regular.
%\end{theo}

%\begin{proof}
%Observe that, if $A$ is associative (in any sense of Section 3),
%then, for any $b$ and $b'$, the equality (4.1) is satisfied for some
%$a_{1},a_{2},...,a_{n}$, $a'_{1},a'_{2},...,a'_{n}\in A$. Indeed, we
%have
%$$\theta(a_{1},a_{2},...,a_{n},b)=\theta(a_{1},a_{2},...,a_{n},\theta(\alpha_{1}(b,b'),\alpha_{2}(b,b'),...,\alpha_{n}(b,b'),b')).$$

%\noindent Applying an associativity condition, we obtain (4.1). At
%that $a'_{1},a'_{2},$ ...,$a'_{n}$ are terms of
%$a_{1},a_{2},...,a_{n}$, and
%$\alpha_{1}(b,b'),\alpha_{2}(b,b'),...,\alpha_{n}(b,b')$. Therefore,
%the condition (i) of Proposition 4.2 implies the condition (ii) of
%the same proposition. The same arguments prove that the condition
%(i) is equivalent to the condition of this theorem.
%\end{proof}
\vskip+3mm

%Proposition 3.2 and Lemma 4.7 immediately implies

%\begin{cor}
%Let $A$ be a $\mathbb{V}$-algebra satisfying any of associativity
%conditions introduced in Section 3 and the following identity:
%\begin{equation}
%\alpha_{i}(\theta(a_{1},a_{2},...,a_{n},b),b')=\theta(a_{1},a_{2},...,a_{n},\alpha_{i}(b,b')),
%\end{equation}
%\noindent for any $i$. Then any $T_{0}$ topology on $A$ is
%completely regular.
%\end{cor}\vskip+2mm

%\begin{exmp}
%(a) The set of non-zero rational numbers with the structure of a
%semi-abelian algebra described in Example 3.8(a) is 1-associative
%and satisfies the condition of Theorem 5.1. Indeed, $r_{1}$ can be
%taken as $\theta (a_{1},a_{2},a_{3})$, since (5.2) is satisfied for
%$i=1$. For $i=2$, $r_{i}$ can be taken as the constant
%$1$.\vskip+2mm

%(b) For any category with finite products and its any object $C$,
%the set $Mor_{\triangle}(C^{n},C)$ of retracts of the diagonal
%$\triangle$ is a 2-associative protomodular algebra and satisfies
%(5.2). \vskip+2mm

%\end{exmp}

\section{Acknowledgment}
This work is supported by Shota Rustaveli National Science
Foundation of Georgia (Ref.: FR-18-10849).

\vskip+2mm

Authors address:

Andrea Razmadze Mathematical Institute of Tbilisi State University,

6 Tamarashvili Str., Tbilisi, 0177, Georgia

E-mail: dalizan@rmi.ge

\end{document}